\documentclass[11pt]{article}
\usepackage{amsmath,amssymb,amsthm,amstext}
\usepackage{fullpage}
\newtheorem{theorem}{Theorem}[section]
\newtheorem{lemma}[theorem]{Lemma}
\newtheorem{corollary}[theorem]{Corollary}
\newtheorem{example}{Example}[section]
\newtheorem*{remark}{Remark}
\newtheorem{construction}{Construction}

\newcommand{\B}{\ensuremath{\mathcal{B}}}
\newcommand{\lam}{\ensuremath{\lambda}}
\newcommand{\C}{\ensuremath{\mathcal{C}}}
\newcommand{\D}{\ensuremath{\mathcal{D}}}

\begin{document}

\title{A new look at an old construction: constructing (simple) $3$-designs from resolvable $2$-designs}
\author{Douglas~R.~Stinson\thanks{D.~Stinson's research is supported by 
NSERC discovery grant 203114-11}
\\David R. Cheriton School of Computer Science\\ University of Waterloo\\ 
Waterloo, Ontario, N2L 3G1, Canada
\and
Colleen M.~Swanson\\
Computer Science \& Engineering Division\\
University of Michigan\\
Ann Arbor, MI 48109, USA
\and
Tran van Trung\\
Institut f\"{u}r Experimentelle Mathematik \\
Universit\"{a}t Duisburg-Essen \\
Ellernstrasse 29, 
45326 Essen, Germany}

\date{\today}

\maketitle

\begin{abstract}
In 1963, Shrikhande and Raghavarao \cite{SR63} published a recursive construction for designs that  
starts with a resolvable design (the ``master design'') and then uses a second design
(``the indexing design'')  to take certain unions of
blocks in each parallel class of the master design. Several variations of this construction have
been studied by different authors. We revisit this construction, concentrating on the case where the
master design is a resolvable BIBD and the indexing design is a 3-design.
We show that this construction yields a $3$-design under certain circumstances. 
The resulting $3$-designs have
block size $k = v/2$ and they are resolvable.
We  also  construct some previously unknown simple designs by this method.
\end{abstract}

\section{Introduction}

We begin with a few definitions.
Suppose that $v,b,r$ and $k$ are positive integers such that
$2 \leq k < v$. 
A \emph{$(v,b,r,k)$-incomplete block design}  (or \emph{$(v,b,r,k)$-IBD})
is a pair $(X, \B)$, where $X$ is a set of 
$v$ \emph{points} and $\B$ is a collection (i.e., a multiset) of $b$ subsets of $X$ (called \emph{blocks}),
such that 
\begin{enumerate}
\item every block contains exactly $k$ points, and 
\item every  point is contained in exactly $r$ blocks.
\end{enumerate}
It follows that $bk=vr$ in a $(v,b,r,k)$-IBD.
Here are a few definitions relating to additional properties that IBDs might satisfy.
\begin{itemize}
\item Suppose $k \mid v$. 
A \emph{parallel class} in a $(v,b,r,k)$-IBD $(X, \B)$ is a set of $v/k$ disjoint blocks.
$(X, \B)$ is \emph{resolvable} if $\B$ can be partitioned into $r$ parallel classes.
\item
A $(v,b,r,k)$-IBD $(X, \B)$ is \emph{simple} if $\B$ does not contain any
repeated blocks, i.e., $\B$ is a set. 
\item
A $(v,b,r,k)$-IBD $(X, \B)$ is \emph{trivial} if $\B$ 
consists of all $\binom{v}{k}$ possible $k$-subsets of points.
\end{itemize}

A \emph{$(v, b, r, k, \lam)$-balanced incomplete block
design} (or a \emph{$(v, b, r, k, \lam)$-BIBD}) is a $(v, b, r, k)$-IBD 
in which every pair of points occurs in exactly $\lam$ blocks. 
It is necessarily the case that 
$\lam = r(k-1)/(v-1)$.
Sometimes the notation $(v, k, \lam)$-BIBD is used, since the
parameters $r$ and $b$  can be computed from $v,k$ and $\lam$ 
(specifically, $r = \lambda (v-1)/(k-1)$ and $b = vr/k$).

Suppose that $\lam, t,k,$ and $v$ are positive integers such that
$t \leq k < v$. 
A \emph{$t$-$(v,k,\lam)$-design} is a pair $(X, \B)$, where $X$ is a set of 
$v$ \emph{points} and $\B$ is a collection (i.e., a multiset) of $k$-subsets of $X$ (called \emph{blocks}),
such that every subset of $t$ points occurs in exactly $\lambda$ blocks.
It is well-known that a subset of $j\leq t-1$ points occurs in $\lambda_j$ blocks in $\B$, where
\begin{equation}
\label{lambda-j.eq} \lambda_j = \frac{\lambda \binom{v-j}{t-j}}{\binom{k-j}{t-j}}.
\end{equation}

If we  define $b = \lambda_0$ and $r = \lambda_1$, then a $t$-$(v,k,\lam)$-design
is a $(v,b,r,k)$-IBD. A $2$-$(v,k,\lam)$-design is the same thing as a $(v, k, \lam)$-BIBD.
Also, a $1$-$(v,k,r)$-design is equivalent to a $(v,vr/k,r,k)$-IBD.  A resolvable $2$-$(v,2,1)$-design is equivalent to a \emph{one-factorization} of the complete graph
$K_v$. The parallel classes in this design are called \emph{one-factors}.

\subsection{Related Work}

There are some examples of  related work, which we mention  now.
There is an unpublished construction due to Lonz and Vanstone  of
$3$-$(v,4,3)$-designs from resolvable $2$-$(v,2,1)$-designs (see \cite{JV86}). 
They take all possible
unions of two blocks from a parallel class of a $2$-$(v,2,1)$-design; this
immediately yields a $3$-$(v,4,3)$-design. In \cite{PSV89},
this method was extended to show the existence of simple $3$-$(v,4,3)$-designs.
Recently, this approach was generalized by Jimbo, Kunihara, Laue and Sawa in \cite{JKLS11}. 
The master design is
a resolvable $3$-$(v,k,\lam)$-design and the indexing design is a $3$-design on
$v/k$ points. The construction \cite{JKLS11} can also be applied when the master design is a
one-factorization (this is a trivial $2$-design that can be viewed as a
degenerate $3$-design with $\lam = 0$). 

We employ the same construction as in \cite{JKLS11}, except the master design is a resolvable 
$2$-design instead of a $3$-design. We still end up with a
$3$-design provided that the other ingredient in the construction is
a $3$-design on $v/k$ points with block size $v/(2k)$. All of our constructions, 
as well as the other constructions mentioned above, can be considered to be 
applications of the Shrikhande-Raghavarao construction.

\subsection{Organization of the Paper}

In Section \ref{construction.sec}, we present the main construction and discuss when this
construction will yield (resolvable) $3$-designs. 
Section \ref{nontrivial.sec} presents conditions that guarantee that the main construction will yield a nontrivial design.
In Section \ref{simple.sec},
we investigate the simplicity of the constructed designs. We give some conditions that will guarantee
that the constructed designs are simple. We also look at several specific constructed 
3-designs and show by computer that they
are simple even though we cannot prove theoretically that they are.
Simple 3-designs for these parameters were not previously known to exist.
Finally, Section \ref{conclusion.sec} is a brief conclusion.

\section{The Construction}
\label{construction.sec}

We begin by describing a  construction of Shrikhande and Raghavarao published in 1963 \cite{SR63}.

\begin{construction}
\label{main}
We need two ingredients for the construction:
\begin{enumerate}
\item 
Let $(X, \B)$ be a resolvable $(v, b, r, k)$-IBD and let $\Pi_1, \ldots, \Pi_r$ 
denote the parallel classes in the resolution of $(X, \B)$. 
There are $w = v/k$ blocks in each parallel class. 
Let the blocks in $\Pi_i$ be named $B_i^j$, $1 \leq j \leq w$.
We will call $(X,\B)$ the \textbf{master design}.
\item
Suppose $(Y,\C)$ is a $(w, b', r', k')$-IBD,
where $Y = \{ 1, \dots , w\}$.
We will call $(Y,\C)$ the \textbf{indexing design}.
\end{enumerate}
Now, for each $i$, $1 \leq i \leq r$, and for each $C  \in \C$, 
define
\[ D_{i,C} = \bigcup _{j \in C} B_i^j .\]
That is, for every block $C$ of the indexing design and for every parallel class $\Pi_i$ of
the master design, we construct a block $D_{i,C}$ by taking the union of the blocks in $\Pi_i$ indexed 
by $C$. Finally, let \[ \D = \{ D_{i,C} : 1 \leq i \leq r, C  \in \C \}.\]
$(X, \D)$ is the \textbf{constructed design}.
\end{construction}


We will study the design $(X,\D)$ obtained from 
Construction \ref{main}. 
The following easily proven result was first shown in \cite{SR63}.

\begin{theorem}
\label{ibd}
Suppose that $(X, \B)$ is a resolvable $(v,b,r,k)$-IBD,
and suppose $(Y,\C)$  
is a $(w, b', r', k')$-IBD, where $w = v/k$.
Let $(X, \D)$ be defined as in Construction \ref{main}.
Then the constructed design $(X, \D)$ is 
a $(v, b'', r'', k'')$-IBD, where 
\begin{eqnarray*}
b'' &=& r b',\\
r'' &=& r r', \mbox{and} \\
k'' &=& k k'.
\end{eqnarray*}
\end{theorem}

If the two input designs are BIBDs, then 
so is the constructed design. The following result is also from \cite{SR63}.

\begin{theorem}
\label{bibd}
Suppose that $(X, \B)$ is a resolvable $(v,k,\lam)$-BIBD,
and suppose $(Y,\C)$  
is a $(w,k',\lam'_2)$-BIBD, where $w = v/k$.
Let $(X, \D)$ be defined as in Construction \ref{main}.
Then the constructed design $(X, \D)$ is 
a $(v, b'', r'', k'', \lambda'')$-BIBD, where $b'', r''$ and $k''$ are 
the same as in Theorem \ref{ibd}, and 
\[ \lambda'' = \lambda r' + (r- \lambda)\lambda_2'.\]
\end{theorem}

We are interested in determining conditions
which guarantee that the constructed design is a $3$-design. For this 
analysis, we will assume that the indexing design $(Y,\C)$  
is a $3$-$(w,k',\lam')$-design, which requires  $k' \geq 3$.
%

Let $x,y,z \in X$ be three distinct points and suppose these points occur in $\alpha$ blocks of 
the master design $(X, \B)$. 
There are $3(\lambda - \alpha)$ parallel classes in which exactly 
two points of $\{x, y, z\}$ occur in a block, and $r - \alpha - 3(\lambda - \alpha)$ 
parallel classes in which $x$, $y$, and $z$ occur in different blocks. 

For a set of three points $x,y,z \in X$, let $\lambda_{x,y,z}$ denote the number of
blocks in the constructed design $\D$ that contain $x,y$ and $z$.
The following lemma is obtained by simple counting.
\begin{lemma}
\label{lem1}
Suppose that $(X, \B)$ is a resolvable $(v,k,\lam)$-BIBD and $(Y,\C)$  
is a $3$-$(w,k',\lam')$-design where $w=v/k$.  
Let $(X, \D)$ be defined as in Construction \ref{main}.
Suppose that three points
$x,y,z$ occur in exactly $\alpha$ blocks in $\B$. Then,
\begin{align}\label{lambda.3.2}
\lambda_{x,y,z} = \alpha r' + 3 \left (\lambda-\alpha \right )\lambda_2' + \left (r-\alpha -3(\lambda-\alpha) \right)\lambda'.
\end{align}
\end{lemma}

\begin{theorem}
\label{thm1}
Suppose that $(X, \B)$ is a resolvable $(v,k,\lam)$-BIBD and $(Y,\C)$  
is a $3$-$(w,k',\lam')$-design where $w=v/k$.  
Let $(X, \D)$ be defined as in Construction \ref{main}.
Then $(X, \D)$ is a $3$-design if and only if one of the following conditions is
satisfied:
\begin{enumerate}
\item $(X, \B)$ is a $3$-design, 
\item $k=2$, or 
\item $k' = v/(2k)$.
\end{enumerate}
\end{theorem}

\begin{proof}
Since $k' > 2$, if we substitute (\ref{lambda-j.eq}) into (\ref{lambda.3.2}), we see that 
\begin{equation}
\lambda_{x,y,z} = \lambda' \left( 3\lambda \frac{w-2}{k'-2} +  r-3\lambda \right)
 + \lambda' \alpha  \left(\frac{(w-1)(w-2)}{(k'-1)(k'-2)} - 3 \frac{w-2}{k'-2} +2  \right).
 \label{lambda.3.2.1}
\end{equation}
Therefore $\lambda_{x,y,z}$ 
is of the form $c_1 + \alpha c_2$, for constants $c_1$ and $c_2$. 
Clearly, the design $(X, \D)$ is a $3$-design if and only if $\alpha$ is a constant or  $c_2 = 0$. 
There are three possible cases that can occur:
\begin{enumerate}
\item If $\alpha>0$ is a constant, this implies that the master design $(X, \B)$ is a $3$-design.
In this situation $(X, \D)$ 
will be a $3$-design for any choice of $k' > 2$.
\item If $\alpha=0$, then it must be the case that $k=2$ 
(here, the master design can be thought of as a ``degenerate'' $3$-design). Again,
$(X, \D)$ 
will be a $3$-design for any choice of $k' > 2$.
\item 
We now use (\ref{lambda.3.2.1}) to derive the conditions that yield $c_2 = 0$.
Because $k' > 2$, 
we have
\begin{align*}
& \frac{(w-1)(w-2)}{(k'-1)(k'-2)} - 3 \frac{w-2}{k'-2} +2  = 0  \\
\iff & (w-1)(w-2) - 3(w-2)(k'-1) + 2(k'-1)(k'-2) = 0 \\
\iff & w^2  - 3wk'  + 2(k')^2  = 0 \\
\iff & (w-k')(w-2k') = 0.
\end{align*}
That is, $c_2=0$ when $k' = v/k$ or $k' = v/(2k)$. Note that if $k' = v/k$, 
then $(X, \D)$ is not a $3$-design since each block in $\B$ would consist of all the points of $X$. So 
$(X, \D)$ is a $3$-design if and only if $k' = v/(2k)$, 
in which case $2k \mid v$. 
\end{enumerate}
\end{proof}

The first case in Theorem \ref{thm1}, where the master design is a $3$-design, is the case that was studied
in \cite{JKLS11}. The second case, where $k=2$, was also studied
in \cite{JKLS11}. 
We will study the third case in more detail. First, we record the designs
that are constructed in this case. 

\begin{corollary}
\label{cor}
Suppose the following designs exist:
\begin{enumerate}
\item a resolvable $(v, b, r, k, \lambda)$-BIBD, and
\item a $3$-$(w,w/2,\lambda')$ design, where $w = v/k >4$ is even.
\end{enumerate}
Then there exists a $3$-$(v,v/2, \mu)$ design, where
\begin{equation}
\label{lambda3} \mu = 
\lambda' \left(  \frac{3\lambda w}{w-4} +  r \right)\end{equation}
\end{corollary}

\begin{proof}
The value of $\mu$ can be obtained from (\ref{lambda.3.2.1}), where the coefficient of
$\alpha$ is equal to $0$. We have
\[ \mu = 
\lambda' \left( 3\lambda \frac{w-2}{\frac{w}{2}-2} +  r-3\lambda \right) = 
\lambda' \left(  \frac{3\lambda w}{w-4} +  r \right).\]
\end{proof} 

Theorem \ref{thm1} requires that $k' > 2$.
It is also useful to consider the variant 
where $k' = 2$; in this case, we can take $(Y,\C)$ to be a (trivial) $2$-$(w,2,1)$-design. 
As in Lemma \ref{lem1}, suppose that
three points
$x,y,z$ occur in exactly $\alpha$ blocks in $\B$.
 From (\ref{lambda.3.2}) with $\lambda'=0$, we have that
\begin{align}\label{lambda.3.1}
\lambda_{x,y,z} = \alpha (w-1) + 3 \left (\lambda-\alpha \right ) = \alpha (w-4) + 3 \lambda.
\end{align}
If $\alpha$ is not constant, it is clear that $\lambda_{x,y,z}$ will be constant if and only if
$w=4$. In this case, the indexing design is a $2$-$(4,2,1)$-design, which is 
 a one-factorization of $K_4$.
Therefore, we have the following variant of Corollary \ref{cor}.

\begin{corollary}
\label{corw4}
Suppose there is a resolvable $(v, b, r, k, \lambda)$-BIBD, where 
$v/k = 4$. Then there is a $3$-$(v,v/2,3\lambda)$-design.
\end{corollary}

\subsection{Resolvability of the Constructed Designs}

In this section, we consider when the constructed design will be resolvable.
This question is easy to answer; we state the following simple result without proof.

\begin{lemma}
The constructed design $(X, \D)$ in Construction \ref{main} is resolvable whenever the 
indexing design $(Y,\C)$ is resolvable.
\end{lemma}

However, if we look at Corollaries \ref{cor} and \ref{corw4}, we can proceed in 
a slightly different manner. Both of these corollaries yield $3$-designs with block
size $v/2$. A result of Tran \cite[Theorem 2.7]{Tr01} shows that a 
resolvable $3$-$(v,v/2, \mu)$-design exists whenever a $3$-$(v,v/2, \mu)$-design exists.
Thus we have the following extensions of Corollaries \ref{cor} and \ref{corw4}.




\begin{corollary}
\label{existence}
Suppose the following designs exist:
\begin{enumerate}
\item a resolvable $(v, b, r, k, \lambda)$-BIBD, and
\item a $3$-$(w,w/2,\lambda')$ design, where $w = v/k >4$ is even.
\end{enumerate}
Then there exists a resolvable $3$-$(v,v/2, \mu)$ design, where
\[\mu = \lambda' \left(  \frac{3\lambda w}{w-4} +  r \right).\]
\end{corollary}

\begin{corollary} 
\label{existence2}
Suppose there is a resolvable $(v, b, r, k, \lambda)$-BIBD where 
$v/k = 4$. Then there is a resolvable $3$-$(v,v/2,3\lambda)$-design.
\end{corollary}


\section{Nontriviality of the Constructed Designs}
\label{nontrivial.sec}

It is easy to see that Construction \ref{main} will often  
produce nontrivial designs. Here is a small example to illustrate.

\begin{example}
{\rm Suppose we start with the trivial $(18,816,136,3)$-IBD consisting of all the
$3$-subsets of an $18$-set. This design 
is resolvable into $r = 136$ parallel classes by Baranyi's Theorem. Apply Construction \ref{main} where the indexing 
design is a $(6,20,10,3)$-IBD consisting of all the
$3$-subsets of a $6$-set.
The constructed design contains $2720$ blocks. 
However, a trivial IBD with $18$ points and block size $9$ 
contains  $\binom{18}{9} = 48620$ blocks. Therefore the constructed design is
nontrivial. 
\qed
}
\end{example}

More generally, we will show that the constructed design 
will be nontrivial whenever the master design and indexing design
are both  \emph{simple} designs.
This is because a trivial design with block size $v/2$ contains  $\binom{v}{v/2}$ blocks, 
and the number of blocks in the constructed design will turn out to be smaller than that.
%
%
%
Suppose $2k$ divides $v$ and $v/k > 4$. 
We apply Construction \ref{main} with the master design being a simple  
design. The number of blocks in the master design is at
most $\binom{v}{k}$ and therefore $r \leq \binom{v-1}{k-1}$.   
Since the indexing design is simple, it follows that the number of blocks in the constructed design
is at most
\[ r \binom{\frac{v}{k}}{\frac{v}{2k}} \leq  \binom{v-1}{k-1} \binom{\frac{v}{k}}{\frac{v}{2k}}.\]
If we can prove the inequality
\begin{equation}
\label{NT.eq}
\binom{v-1}{k-1} \binom{\frac{v}{k}}{\frac{v}{2k}} < \binom{v}{\frac{v}{2}},
\end{equation}
then we will have shown that the constructed design is nontrivial. To do this, we will make use of the
well-known inequalities 
\begin{equation}
\label{bin1.eq}
\binom{v-1}{k-1} < \left( \frac{v}{k} \right)^{k-1}
\end{equation}
and 
\begin{equation}
\label{bin2.eq}
\frac{4^n}{2 \sqrt{n}} < \binom{2n}{n} < \frac{4^n}{\sqrt{\pi n}}.
\end{equation}
From (\ref{bin1.eq}) and (\ref{bin2.eq}), it is clear that 
(\ref{NT.eq}) will hold provided that we can prove that the following inequality holds:
\begin{equation}
\label{NT2.eq}
\left( \frac{v}{k} \right)^{k-1} \frac{4^{v/2k}}{\sqrt{\pi (v/2k)}} < \frac{4^{v/2}}{2 \sqrt{v/2}}.
\end{equation}
After some simplification, it is easy to see that (\ref{NT2.eq}) is equivalent to
\begin{equation}
\label{NT3.eq}
\left( \frac{v}{k} \right)^{k-1}  < 2^{v(k-1)/k} \sqrt{\frac{\pi}{4k}}.
\end{equation}
The inequality (\ref{NT3.eq}) can be rewritten as
\begin{equation}
\label{NT4.eq}
\left( \frac{2^{v/k}}{  \frac{v}{k}} \right)^{k-1}  >  \sqrt{\frac{4k}{\pi}}.
\end{equation}
Since $v/k \geq 4$, the left side of (\ref{NT4.eq}) is at least $4^{k-1}$.
Since $k \geq 2$, it is easy to see that  $4^{k-1}\geq 2k$.
Finally it is obvious that $2k > \sqrt{\frac{4k}{\pi}}$. 
Therefore (\ref{NT4.eq}) holds, and we have proven the following.

\begin{theorem}
\label{nontrivial}
Suppose that $2k$ divides $v$ and $v/k \geq 4$. 
Suppose  we apply Construction \ref{main} where the master and indexing designs
are both simple designs.
Then the constructed design  is nontrivial.
\end{theorem}

\section{Simple Designs}
\label{simple.sec}

In certain situations, Corollaries \ref{cor} and \ref{corw4}  yield simple designs.
This is the theme we pursue in this section.

Suppose $(X, \B)$ is a resolvable $(v, b, r, k)$-IBD and let $w = v/k$.
Suppose the assumed resolution of  $(X, \B)$ consists of 
$r$ parallel classes $\Pi_1, \ldots, \Pi_r$,
where $\Pi_i = \{ B_i^1, \dots , B_i^w\}$ for $1 \leq i \leq r$.
Let $1 \leq \alpha \leq w-1$ be an integer.
We say that two parallel classes $\Pi_i$ and $\Pi_j$ satisfy the 
\emph{$\alpha$-partial replacement property} (or \emph{PRP}) if
there exist two parallel classes $\Pi_i^*$ and $\Pi_j^*$ of blocks in $\B$ 
such that
\begin{equation}
\label{PRP-eq}
\Pi_i^* \cup \Pi_j^* = \Pi_i \cup \Pi_j \quad \text{and} \quad 
|\Pi_i^* \cap \Pi_i| = \alpha.
\end{equation} We say that the resolution of $(X, \B)$
is \emph{$\alpha$-PRP-free} if there do not exist two parallel classes 
in $\{ \Pi_1, \ldots, \Pi_r \}$ 
that satisfy $\alpha$-PRP. Further, we say that the resolution of $(X, \B)$
is \emph{PRP-free} if there do not exist two parallel classes 
in $\{ \Pi_1, \ldots, \Pi_r \}$ 
that satisfy $\alpha$-PRP for any positive integer $\alpha < w$.

As an example of a resolvable BIBD that is \emph{not} PRP-free,
consider a one-factorization of $K_{4n}$ that contains a
sub-one-factorization of $K_{2n}$. Each one-factor of a $K_{4n}$ has $2n$ edges 
and each one-factor of a $K_{2n}$ contains $n$ edges.
There are $2n-1$ one-factors of the $K_{4n}$ 
that contain a one-factor of the $K_{2n}$. The $n$ edges in any one-factor of
$K_{2n}$ may be swapped with the $n$ edges in any other one-factor of $K_{2n}$.
Therefore, the two corresponding one-factors in the one-factorization of $K_{4n}$ satisfy  $n$-PRP.

\begin{lemma}
\label{unique}
If a $(v,b,r,k)$-IBD has a unique resolution, then it is PRP-free.
\end{lemma}

\begin{proof} Let $w = v/k$. Suppose there are two parallel classes $\Pi_i$ and $\Pi_j$
that satisfy $\alpha$-PRP for some $\alpha$, where $1 \leq \alpha \leq w-1$.
 Let $\Pi_i^*$ and $\Pi_j^*$ be the two parallel classes that
satisfy (\ref{PRP-eq}). If we replace $\Pi_i$ and $\Pi_j$ by $\Pi_i^*$ and $\Pi_j^*$,
then we obtain a second resolution of the given design.
\end{proof}

\begin{theorem}
\label{PRP}
Suppose we apply Construction \ref{main} with the master design having a resolution
that is $k'$-PRP-free, where the indexing design has block size $k'$.
Then the constructed design is simple.
Furthermore, in the case where the indexing design is trivial, the constructed design is
simple if and only if the resolution of the master design is $k'$-PRP-free.
\end{theorem}

\begin{proof}
Suppose that $(X,\D)$ is not simple; then $D_{i,C} = D_{i',C'}$
where $(i,C) \neq (i',C')$. If $i = i'$ and $D_{i,C} = D_{i',C'}$,
then $C = C'$. Therefore, we can assume $i \neq i'$.
If we define 
\[ \Pi_i^* = \{ B_i^j : j \in C\}  \cup \{ B_{i'}^j : j \not\in C'\} \] and
\[ \Pi_{i'}^* = \{ B_{i'}^j : j \in C'\}  \cup \{ B_{i}^j : j \not\in C\} ,\]
 then we see that
the parallel classes $\Pi_i$ and $\Pi_{i'}$ in $(X,\B)$ satisfy $k'$-PRP.

Now, assume that the indexing design is trivial and there are two parallel classes
$\Pi_i$ and $\Pi_{i'}$ in the resolution of $(X,\B)$ that satisfy $k'$-PRP. So there exist 
$\Pi_i^*$ and $\Pi_{i'}^*$ such that 
\[\Pi_i^* \cup \Pi_{i'}^* = \Pi_i \cup \Pi_{i'} \quad \text{and} \quad 
|\Pi_i^* \cap \Pi_i| = k'.\]
Let $C = \{j: B_i^j \in \Pi_i^* \cap \Pi_i\}$ and $C' = \{j: B_{i'}^j \in \Pi_{i'}^* \cap \Pi_{i'}\}$.
Note that $C$ and $C'$ are both blocks in the indexing design, because the indexing design is trivial.
Then it is easy to see that $D_{i,C} = D_{i',C'}$, so the constructed design is not simple.
\end{proof}

We give a well-known class of BIBDs having unique resolutions.

\begin{lemma}
\label{ARBIBD}
For all prime powers $q$ and for all integers $m \geq 2$, there exists
a PRP-free resolvable $\left( q^m, q^{m-1}, (q^{m-1}-1) / (q-1) \right)$-BIBD.
\end{lemma}

\begin{proof}
The hyperplanes of the
$m$-dimensional affine geometry $\mathsf{AG}(m,q)$  %
over $\mathbb{F}_{q}$ yield a 
resolvable $\left( q^m, q^{m-1}, (q^{m-1}-1) / (q-1) \right)$-BIBD
where each parallel class consists of $q$ blocks.
Furthermore, any two blocks from different parallel classes intersect in 
exactly $q^{m-2}$ points. From this fact, it is easily seen that this design has a unique resolution.
Hence, from Lemma \ref{unique},  the resolution is PRP-free.
\end{proof}

\begin{remark}
Any affine resolvable BIBD has a unique resolution and therefore is PRP-free.
\end{remark}

\begin{corollary}
Suppose  $q = 2^n > 4$ and there exists 
a $3$-$(q,q/2,\lambda')$ design.
Then there exists a simple  $3$-$(q^m,q^{m-1}, \mu)$ design, where
\[ \mu = 
\lambda' \left(\frac{q^m-4}{q-4}\right) .\]
\end{corollary}

\begin{proof}
We apply Theorem \ref{PRP},
starting with the resolvable $\left( q^m, q^{m-1}, (q^{m-1}-1) / (q-1) \right)$-BIBD
from Lemma \ref{ARBIBD}. This design has 
$r = (q^{m}-1) / (q-1)$ and $w = q$. Corollary \ref{cor} establishes that the constructed design will
be a $3$-$(q^m,q^{m-1}, \mu)$-design, where 
\begin{eqnarray*}
\mu &=&\lambda' \left(  \frac{3\lambda w}{w-4} +  r \right)\\
& = & \lambda' \left(  \frac{3(q^{m-1}-1)}{q-1} \times \frac{q}{q-4} +  \frac{q^{m}-1}{q-1} \right)\\
& = & \lambda' \left( \frac{3(q^{m}-q) + (q^{m}-1)(q-4)}{(q-1)(q-4)} \right)\\
& = & \lambda' \left( \frac{q^{m+1} - q^m - 4q + 4}{(q-1)(q-4)} \right)\\
& = & \lambda' \left( \frac{q^m -  4}{q-4} \right).
\end{eqnarray*}
Finally, Theorem \ref{PRP} shows that the constructed design is simple.
\end{proof}

As an example, suppose we take $q = 8$ and $m=2$. Here, the master design is the affine plane of  order
$8$, i.e., a resolvable $(64,72,9,8,1)$-BIBD. If the indexing design is
a $3$-$(8,4,1)$-design, then the constructed
$3$-$(64,32,15)$-design is simple. If the indexing design is the trivial 
$3$-$(8,4,5)$-design, then the constructed
$3$-$(64,32,75)$-design is again simple.

It also turns out that many  $3$-designs constructed from Corollary
\ref{existence} are simple, even if we
cannot prove theoretically that they are. In fact, we have checked several 
designs by computer and verified computationally that they are simple $3$-designs.
These designs all arise from trivial
indexing designs. They are listed in Table \ref{tab2} and details are provided in the rest of this section.

\begin{table}
\caption{Simple $3$-designs constructed from Corollaries
\ref{cor} and \ref{corw4}}
\label{tab2}
\begin{center}
\begin{tabular}{c|c|c}
master design &   $w$ & constructed design\\ \hline
$2$-$(24,4,3)$  & 6 & $3$-$(24,12,50)$\\
$2$-$(24,6,5)$  & 4 & $3$-$(24,12,15)$\\
$2$-$(24,3,2)$  & 8 & $3$-$(24,12,175)$\\
$2$-$(28,7,6)$  & 4 & $3$-$(28,14,18)$\\
$2$-$(30,5,4)$  & 6 & $3$-$(30,15,65)$\\
$2$-$(30,3,2)$  & 10 & $3$-$(30,15,819)$\\
$2$-$(32,8,7)$  & 4 & $3$-$(32,16,21)$\\
$2$-$(36,9,8)$  & 4 & $3$-$(36,18,24)$
\end{tabular}
\end{center}
\end{table}

\subsection{A Simple $3$-$(24,12,15)$-design}
We construct a simple $3$-$(24,12,15)$-design on
point set $X=\{0,1, \dots, 22, \infty\}$. The cyclic group $\mathbb{Z}_{23}$
acts on the design. $\mathbb{Z}_{23}$ permutes
the points $\{0,1, \dots ,22\}$ (cyclically) and fixes the point $\infty$.

We start with a resolvable  $(24,6,5)$-BIBD having a cyclic
automorphism of order 23 (see \cite[p.\ 408]{CD06}). 
The four base blocks of the BIBD form a parallel class:
  \[   \{\infty,0,1,7,15,20\},  \{2,3,4,5,10,14\},
 \{6,11,13,17,19,22\}, \{8,9,12,16,18,21\} .\]

Here are the six base blocks of the constructed $3$-$(24,12,15)$ design:
\[
\begin{array}{l} B_1=\{\infty,0,1,7,15,20, 2,3,4,5,10,14\}\\
 B_2=\{\infty,0,1,7,15,20, 6,11,13,17,19,22\}\\
 B_3=\{\infty,0,1,7,15,20, 8,9,12,16,18,21\}\\
 B_4=\{2,3,4,5,10,14, 6,11,13,17,19,22\}\\
 B_5=\{2,3,4,5,10,14, 8,9,12,16,18,21\}\\
 B_6=\{6,11,13,17,19,22, 8,9,12,16,18,21\}.
 \end{array}
 \]
The other blocks are obtained by developing the base blocks through $\mathbb{Z}_{23}$.

The next line gives data about intersection numbers of the design: 
\[69   , 0  , 46  ,  0,  506 ,2208, 3864 ,2208  ,506  ,  0   ,46   , 0  ,  0.\]

There are 13 numbers corresponding to 13 possible intersection numbers,
namely $0,1,2, \dots, 12$.
For example, 69 pairs of blocks have zero intersection; they form
the parallel classes of the design.
The second value equals 0, which says that there are no pairs of blocks
intersecting in 1 point.
The third value equals 46 pairs, which says that there are 46 pairs of blocks intersecting 
in 2 points, etc.
The data show that any two blocks of the design intersect in at most
10 points, so the design is simple.

\subsection{A Simple $3$-$(28,14,18)$-design}

We construct a simple $3$-$(28,14,18)$-design on
point set $X=\{0,1, \dots, 26, \infty\}$. The cyclic group $\mathbb{Z}_{27}$
acts on the design. $\mathbb{Z}_{27}$ permutes
the points $\{0,1, \dots ,26\}$ (cyclically) and fixes the point $\infty$.

We start with a resolvable  $(28,7,6)$-BIBD having a cyclic
automorphism of order 27 (see \cite[p.\ 408]{CD06}). 
The four base blocks of the BIBD form a parallel class:
  \[   \{\infty,0,5,14,15,24,25\},  \{1,10,17,20,22,23,26\},
 \{2,6,8,9,13,19,21\}, \{3,4,7,11,12,16,18\} .\]

Here are the six base blocks of the constructed $3$-$(28,14,18)$ design:
\[
\begin{array}{l} B_1=\{\infty,0,5,14,15,24,25, 1,10,17,20,22,23,26\}\\
 B_2=\{\infty,0,5,14,15,24,25, 2,6,8,9,13,19,21\}\\
 B_3=\{\infty,0,5,14,15,24,25, 3,4,7,11,12,16,18\}\\
 B_4=\{1,10,17,20,22,23,26, 2,6,8,9,13,19,21\}\\
 B_5=\{1,10,17,20,22,23,26, 3,4,7,11,12,16,18\}\\
 B_6=\{2,6,8,9,13,19,21, 3,4,7,11,12,16,18\}.
 \end{array}
 \]
The other blocks are obtained by developing the base blocks through $\mathbb{Z}_{27}$.

The next line gives the intersection numbers of the design: 
\[81 ,   0 ,   0   , 0  , 54, 1080 ,3132, 4428 ,3132 ,1080,   54  ,  0  ,  0   , 0 ,   0 .\]
This data shows that the design is simple.

\subsection{A Simple $3$-$(30,15,65)$-design}

We construct a simple $3$-$(30,15,65)$-design on
point set $X=\{0,1, \dots, 28, \infty\}$. The cyclic group $\mathbb{Z}_{29}$
acts on the design. $\mathbb{Z}_{29}$ permutes
the points $\{0,1, \dots ,28\}$ (cyclically) and fixes the point $\infty$.

We start with a resolvable  $(30,5,4)$-BIBD having a cyclic
automorphism of order 29 (see \cite[p.\ 408]{CD06}). 
The six base blocks of the BIBD form a parallel class:
  \[  \begin{array}{l}   \{\infty,4,8,17,18\},  \{0,1,7,24,27\},
 \{2,12,15,16,23\},\\ \{9,11,13,21,26\},
  \{3,14,20,22,25\}, \{5,6,10,19,28\} .\end{array}\]

Here are the 20 base blocks of the constructed $3$-$(30,15,65)$ design:
\[
\begin{array}{l}    
B_1=\{\infty,4,8,17,18, 0,1,7,24,27, 2,12,15,16,23\} \\
    B_2=\{\infty,4,8,17,18, 0,1,7,24,27, 9,11,13,21,26\} \\
    B_3=\{\infty,4,8,17,18, 0,1,7,24,27, 3,14,20,22,25\} \\
    B_4=\{\infty,4,8,17,18, 0,1,7,24,27, 5,6,10,19,28\} \\
    B_5=\{\infty,4,8,17,18, 2,12,15,16,23, 9,11,13,21,26\} \\
    B_6=\{\infty,4,8,17,18, 2,12,15,16,23, 3,14,20,22,25\} \\
    B_7=\{\infty,4,8,17,18, 2,12,15,16,23, 5,6,10,19,28\} \\
    B_8=\{\infty,4,8,17,18, 9,11,13,21,26, 3,14,20,22,25\} \\
    B_9=\{\infty,4,8,17,18, 9,11,13,21,26, 5,6,10,19,28\} \\
    B_{10}=\{\infty,4,8,17,18, 3,14,20,22,25, 5,6,10,19,28\} \\
       B_{11}=\{0,1,7,24,27, 2,12,15,16,23, 9,11,13,21,26\} \\
    B_{12}=\{0,1,7,24,27, 2,12,15,16,23, 3,14,20,22,25\} \\
    B_{13}=\{0,1,7,24,27, 2,12,15,16,23, 5,6,10,19,28\} \\
    B_{14}=\{0,1,7,24,27, 9,11,13,21,26, 3,14,20,22,25\} 
\end{array}\]
    
  \[  \begin{array}{l} 
     B_{15}=\{0,1,7,24,27, 9,11,13,21,26, 5,6,10,19,28\} \\
    B_{16}=\{0,1,7,24,27, 3,14,20,22,25, 5,6,10,19,28\} \\
    B_{17}=\{2,12,15,16,23, 9,11,13,21,26, 3,14,20,22,25\} \\
    B_{18}=\{2,12,15,16,23, 9,11,13,21,26, 5,6,10,19,28\} \\
    B_{19}=\{2,12,15,16,23, 3,14,20,22,25, 5,6,10,19,28\} \\
    B_{20}=\{9,11,13,21,26, 3,14,20,22,25, 5,6,10,19,28\}.
 \end{array}
 \]
The other blocks are obtained by developing the base blocks through $\mathbb{Z}_{29}$.

The next line gives the intersection numbers of the design: 
\[290  , 0  , 0 ,  58 , 1044 ,10382 ,24940 ,47386 ,47386 ,24940 ,10382  ,1044,  58 , 0 ,  0  , 0.\]
This data shows that the design is simple.

\subsection{A Simple $3$-$(24,12,50)$-design}

A $3$-$(24,12,50)$-design
   constructed from a resolvable $(24,4,3)$-BIBD having a cyclic
    automorphism of order 23 (see \cite[p.\ 407]{CD06}).
The six base blocks of the BIBD form a parallel class:
  \[   \{\infty, 0 ,7 ,10\}, \{ 1, 8 ,12, 22\}, \{2 ,5, 6 ,11\}, 
  \{3 ,9 ,14, 18\}, \{4 ,16 ,17 ,19\}, \{13, 15, 20 ,21\} .\]
The intersection numbers of the $3$-$(24,12,50)$-design (from 0 to 12 respectively) are
\[230,       0 ,      0 ,   1242 ,   9982 ,  23414 ,  36064 ,  23414  ,  9982  ,  1242,
  0   ,    0  ,     0.\]
  
\subsection{A Simple $3$-$(24,12,175)$-design}

A $3$-$(24,12,175)$-design
    constructed from a resolvable $(24,3,2)$-BIBD having a cyclic
    automorphism of order 23 (see \cite[p.\ 407]{CD06}).
    The eight base blocks of the BIBD  form a parallel class: 
    \[   \{   \infty ,16, 20\}, \{0, 7 ,21\}, \{1 ,3 ,11\}, \{4, 5, 18\}, \{6 ,12, 17\}, \{2 ,10, 13\}, 
                \{8 ,9, 14\}, \{15, 19 ,22 \} .\]
The intersection numbers of the $3$-$(24,12,175)$-design (from 0 to 12 respectively) are:
\[  805  ,    46 ,   1380  , 30314 ,  99958 , 289432  ,452180 , 289432  , 99958  , 30314,
 1380   ,   46   ,    0 .\]

\subsection{A Simple $3$-$(30,15,819)$-design}

A $3$-$(30,15,819)$-design 
  constructed from a resolvable $(30,3,2)$-BIBD having a cyclic
    automorphism  of order 29 (see \cite[p.\ 407]{CD06}).
    The ten base blocks of the BIBD form a parallel class: 
  \[  
  \begin{array}{c} 
  \{    \infty, 2 ,22\}, \{ 0 ,1 ,19\}, \{ 6 ,8 ,27\}, \{ 9 ,15 ,16\}, \{ 7, 10 ,14\},\\ 
  \{ 11 ,25 ,28\}, \{ 12 ,18 ,23\}, \{ 13, 21 ,26\}, \{ 4, 20 ,24\}, \{ 3, 5 ,17\}.
  \end{array} 
  \] 
The intersection numbers of the  $3$-$(30,15,819)$ are 
\[
\begin{array}{c} 
3654     ,  0   ,  928 , 115594 , 242730 ,1356272, 4482530 ,7150008, \\
7150008, 4482530 , 1356272 , 242730  ,115594 ,    928  ,     0   ,    0 .
\end{array}\]

\subsection{A Simple $3$-$(32,16,21)$-design}

A  $3$-$(32,16,21)$-design
    constructed from a resolvable $(32,8,7)$-BIBD having a cyclic
    automorphism  of order 31 (see \cite[p.\ 408]{CD06}).
    The four base blocks of the BIBD form a parallel class:
    \[  \begin{array}{c} \{\infty , 0,  1 , 5 ,16 ,18 ,25 ,28\},  \{3 , 9 ,17 ,19 ,20, 21 ,24 ,26\},\\
     \{2 ,6, 7 ,12 ,14 ,15, 23 ,27\},  \{4, 8 ,10 ,11 ,13, 22 ,29, 30) \} .\end{array}\]
The intersection numbers of the $3$-$(32,16,21)$-design are
\[
93  ,  0 ,   0  ,  0   , 0 , 372 , 930, 4836, 4836, 4836,  930 , 372  ,  0  ,  0  ,  0,    0  ,  0.\]

\subsection{A Simple $3$-$(36,18,24)$-design} 

A $3$-$(36,18,24)$-design
constructed from a resolvable $(36,9,8)$-BIBD having a cyclic
    automorphism  of order 35 (see \cite[p.\ 408]{CD06}).
    The four base blocks of the BIBD form a parallel class:
    \[  \begin{array}{c} \{\infty ,1 ,10, 11, 13 ,15, 25, 26 ,29\},    \{0, 9, 16 ,17, 18 ,19, 22, 24 ,30\},\\
         \{ 3 ,4 ,6 ,12 ,21, 23, 27, 31 ,34\},    \{2 ,5 ,7 ,8, 14 ,20 ,28, 32 ,33 \} .\end{array}\]
The intersection numbers of the $3$-$(36,18,24)$-design are
\[105   , 0  ,  0  ,  0 ,   0 ,  70 , 280, 2100, 4970, 7000, 4970 ,2100 , 280   ,70  ,  0   , 0,    0 ,   0,    0 .\]

\section{Conclusion}
\label{conclusion.sec}

There are probably many other possible applications of the constructions in this paper
to obtain (new) simple $3$-designs. It would be very nice to find additional easily checked
conditions that would guarantee that a constructed $3$-design is simple.

\end{document}